\documentclass{elsarticle}
\usepackage{amssymb}
\usepackage{amsmath}
\usepackage{amsthm}
\usepackage{verbatim}
\usepackage{tikz}
\usepackage{pgfplots}
\usepackage{graphicx}
\usepackage{enumitem}
\usepackage{hyperref}
\usetikzlibrary{calc,patterns,shapes,plotmarks,datavisualization}
\newcommand{\ZZ}{\mathbb{Z}}
\newcommand{\NN}{\mathbb{N}}
\newcommand{\RR}{\mathbb{R}}

\newcommand{\CC}{\mathbb{C}}

\DeclareMathOperator{\Bin}{Bin}

\newtheorem{theorem}{Theorem}[section]
\newtheorem*{theorem*}{Theorem}
\newtheorem*{claim*}{Claim}
\newtheorem{prop}[theorem]{Proposition}
\newtheorem{cor}[theorem]{Corollary}
\newtheorem{lem}[theorem]{Lemma}

\newtheorem{conj}[theorem]{Conjecture}
\theoremstyle{definition}
\newtheorem{definition}[theorem]{Definition}
\newtheorem{remark}[theorem]{Remark}
\newtheorem{example}[theorem]{Example}

\title{Random Toric Surfaces and a Threshold for Smoothness}
\fntext[thks]{The author was partially supported by an NSF grant DMS-1502553.\\ \textcopyright 2019. This manuscript version is made available under the CC-BY-NC-ND 4.0 license http://creativecommons.org/licenses/by-nc-nd/4.0/ }
\cortext[cor1]{Corresponding Author}
\address[uw]{Department of Mathematics, University of Minnesota, Minneapolis, MN}
\author[uw]{Jay Yang\corref{cor1}\fnref{thks}}
\ead{yangjay@math.wisc.edu}

\begin{document}
\begin{abstract}
  We present a notion of a \emph{random toric surface} modeled on a notion of a random graph. We then study some threshold phenomena related to the smoothness of the resulting surfaces.
\end{abstract}
\maketitle

\section{Introduction}

When random constructions have been used in other fields, including Graph Theory~\cite{erdos,bollobas} and Topology~\cite{flagcomplex,triangulated}
, they have helped to gain insights into the nature of large graphs that would be difficult to approach via other techniques. And in Algebra, these techniques have yielded a rich theory of Random Groups\cite{randgroup}.

Previous work in Algebraic Geometry on the use of random constructions include uses of random (topological) surfaces to study Gromov-Witten theory~\cite{randsurf}, and more closely related to this paper, work by Winfried Bruns\cite{quest} giving an algorithm for constructing random toric varities for the purpose of constructing counterexamples. In contrast to his technique, we are less concerned with using randomness to exhibit particular counterexamples, rather we will study the distributions given by our construction. This allows us to use tools similar to those in the theory of random graphs to understand toric surfaces as a whole.

In this paper we define a notion of a \emph{random toric surface} using a technique similar to that used to define a random graph. In particular, normal toric surfaces are given by a fan in $\RR^2$, so if we can construct random fans we can construct random toric surfaces. To create a random fan, we start by randomly choosing rays; once we have chosen the rays, in $2$ dimensions, there exists a unique way to ``fill in'' the fan with $2$-cones (see Section~\ref{sec:notation}). As we will see in Definition~\ref{maindef} we will choose rays among those where the first lattice point on them is of magnitude at most $h$. Within this set, we choose whether to use a particular ray with probability $p$, allowing us to form a distribution $T(h,p)$ with two parameters. Here the $h$ parameter is analogous to the number of vertices in a random graph, and we consider the behavior of this distribution as $h\rightarrow \infty$.

We say a fan $\Sigma$ chosen with respect to $T(h,p)$ has a property \emph{with high probability}, if the probability that $\Sigma$ has the property goes to $1$ as $h$ goes to $\infty$. Now we can state our main result. For convenience, define $f\prec g$ for functions $f,g$ to mean $\lim_{h\rightarrow \infty} \frac{f(h)}{g(h)} = 0$.

\begin{theorem}
  \label{mainresult}
  Let $q$ be a function of $h$. Then for a fan $\Sigma$ chosen with respect to $T(h,1-q)$, we have the following behavior
  \begin{enumerate}[ref={\thetheorem(\arabic*)}]
  \item if $q\prec 1/h^2$ or $1-q\prec 1/h^2$ then with high probability $X(\Sigma)$ is smooth \label{mainresultpart1}
  \item if $q\succ 1/h^2$ and $1-q\succ 1/h^2$ then with high probability $X(\Sigma)$ is singular \label{mainresultpart2}
  \end{enumerate}
  In addition for $k>1$, then with $1-q\succ 1/h$ and $q\succ 1/h^2$, we have a singularity of index at least $k$.
\end{theorem}

See Definition~\ref{singindex} for the definition of singularity index and see Conjecture~\ref{stronger} for a slight strengthening. If we ignore the strengthening of the second statement in the case of $1-q\succ 1/h$ , this is analogous to the following familiar result from the theory of random graphs.

\begin{theorem*}[Erd\H{o}s - R\'enyi theorem\cite{erdos}]
  For any $\epsilon>0$ and for a graph $G$ chosen with respect to $G(n,p)$
  \begin{enumerate}
  \item if $p\geq \frac{(1+\epsilon)\log n}{n}$ Then with high probability $G$ is connected
  \item if $p\leq \frac{(1-\epsilon)\log n}{n}$ Then with high probability $G$ is disconnected
  \end{enumerate}
\end{theorem*}

The content of this paper largely revolves around Theorem~\ref{mainresultpart2}. That part itself is proved in two overlapping cases, which we will call the \emph{dense} ($q\succ 1/h^2$ and $1-q\succ 1/h$) and \emph{sparse} case ($1-q\succ 1/h^2$ and $q\rightarrow 1$). This phrasing is motivated by noting that in the first case the number of rays grows at least linearly, and in the second case, the number of rays is always asymptotically less than the set of possible rays. As will be seen, the two cases behave quite differently as to the question of why a random fan is singular. In particular we will have two overlapping cases in our proof which together cover the second part of our theorem.

In contrast, for Theorem~\ref{mainresultpart1} the proof reduces to showing that a particular fan occurs with high probability. In particular, for $q\prec 1/h^2$ we get the complete fan, $\Sigma_h$, as illustrated in Example~\ref{sigmah}. For $1-q\prec 1/h^2$, we get the empty fan, which corresponds to the toric surface $(\CC^{*})^2$.

We should also note that as an immediate consequence of our theorem, we have that for a fixed $p$, random fans are singular.

\begin{cor}
  For fixed $p$ with $0<p<1$ then for a fan $\Sigma$ chosen with respect to $T(h,p)$, with high probability $X(\Sigma)$ is singular.
\end{cor}

We might ask further questions about how many singularities a random toric surface has. As it turns out we can always ensure that there are arbitrarily many of these singularities. But in the case of normal toric surfaces, singularities can only occur at torus fixed points, so we can actually ask a more refined question, namely what proportion of the torus fixed points are actually singular. Let $\delta_k$ be the density of fixed points of singularity index at least $k$ as defined in Section~\ref{sec:density}.

\begin{theorem}
  \label{density}
  Fix some $k>1$, then for a fan $\Sigma$ chosen with respect to $T(h,1-q)$ if $0<q<1$ and $c_k>0$ sufficiently small, then with high probability, $\delta_k(\Sigma)>c_k$.
\end{theorem}

This paper is organized as follows. We start with preliminaries in Section~\ref{sec:notation}. This is followed by a discussion of the geometry of fixed height fans in Section~\ref{sec:geometry}, which includes many of the main definitions. Next we consider what happens when we blow down along rays for these fans in Section~\ref{sec:sing}. These together form the basis of the proofs in Section~\ref{sec:threshold}. We end then with some further ideas and a conjecture in Section~\ref{sec:further}.

\section{Preliminaries}
\label{sec:notation}

We use the following relatively standard notation as seen in \emph{Toric Varieties}\cite{torvar}. $\Sigma$ is a fan, $\sigma$ is a cone in a fan and $\rho$ is a ray in a fan. $u_{\rho}$ is the minimal lattice point on $\rho$. $X(\Sigma)$ is the toric variety corresponding to $\Sigma$.

Recall that a normal toric surface can be given by a rational fan in $\RR^2$. By the orbit-cone correspondence, the $2$-dimensional cones in this fan correspond to fixed points under the torus action. Furthermore, on a normal toric surface, the only singularities occur at the fixed points. Thus our discussion of singularities on a toric surface comes down to discussing singular cones in rational fans in $\RR^2$. For a rational cone in $\RR^2$ given by rays $\rho,\tau$ with corresponding lattice points $u_{\rho}$ and $u_{\tau}$, the cone is smooth if and only if $\left|u_{\rho}\wedge u_{\tau}\right|=1$,  where the the value $\left|u_{\rho}\wedge u_{\tau}\right|$ is the absolute value of the determinant of the $2\times 2$ matrix with columns $u_{\rho}$ and $u_{\tau}$. 

\begin{definition}
  \label{singindex}
  We say that the fixed point on a toric surface corresponding to a cone given by $\rho$ and $\tau$ has \emph{singularity index} $\left|u_{\rho}\wedge u_{\tau}\right|$.
\end{definition}

On a toric surface, for a cone of singularity index k, the singularity is of the form of a quotient by $\ZZ/k\ZZ$. Also, blowing up this singularity gives an exceptional divisor with self intersection number $-k$. The second point will be important, because in Section~\ref{sec:sing} we will use rays of particular self intersection numbers to yield singularities of particular singularity index \cite{torvar}. 

Much of the work here requires that we fix some norm on $\ZZ^2$, for our purposes we will use $\left|(x,y)\right|=\max\{\left|x\right|,\left|y\right|\}$, but essentially all of the proofs can be extended to any norm without any major changes. In addition, as a shorthand for a ray $\rho$, we write $\left|\rho\right|=\left|u_{\rho}\right|$.

\begin{definition}
  \label{completedef}
  Given a finite collection of rays $S$ from $\ZZ^2$, we define the \emph{completion of a set of rays} as the fan $\Sigma$ with $\Sigma(1)=S$ such that $\Sigma$ is maximal w.r.t. inclusion among all fans with this property. 
\end{definition}

At a practical level, this construction simply takes a collection of rays and ``fills in'' every possible cone, as seen in Figure~\ref{fig:fill}. Notice that this construction is completely dependent on $2$ dimensions, and is a large obstacle to the construction of analogous higher dimensional ideas.

\begin{figure}
  \begin{center}
  \begin{tikzpicture}
    \draw[->] (0,0) -- (1,1);
    \draw[->] (0,0) -- (-1,1);
    \draw[->] (0,0) -- (1,0);

    \draw[->,line width=2] (2,0.5) -- (3,0.5);
    
    \fill[gray] (5,0) -- (6,1) -- (4,1) -- cycle;
    \fill[gray] (5,0) -- (6,1) -- (6,0) -- cycle;
    \draw[->] (5,0) -- (6,1);
    \draw[->] (5,0) -- (4,1);
    \draw[->] (5,0) -- (6,0);
    
  \end{tikzpicture}
  \end{center}
  \caption{\label{fig:fill}To fill construction the completion of a set of rays shown above, as described in Definition~\ref{completedef}, we have to add two cones.}
\end{figure}
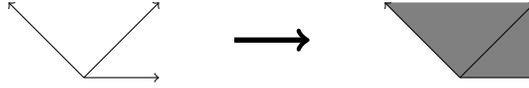

\section{Geometry of Fixed Height Fans}
\label{sec:geometry}

We can now define the distribution of interest.

\begin{definition}
  \label{maindef}
  Let $R$ be the set of rays $\rho\subset \RR^2$ where $\left|\rho\right|\leq h$. We define \emph{the distribution $T(h,p)$} over fans formed by choosing rays in $R$ with probability $p$, and then completing it to a fan.
\end{definition}

\begin{example}
  \label{sigmah}
  In the case where we take the empty set of rays, the construction in Definition~\ref{completedef} gives the empty fan. Recall in particular that the empty fan corresponds to the variety $(\CC^*)^2$.

  In the opposite extreme, where we take the whole set $R$ from Definition~\ref{maindef}, we will denote the resulting fan $\Sigma_h$. As we will see in Lemma~\ref{completesmooth}, this fan is always smooth.
\end{example}

It will turn out to be useful to know that the number of rays of height at most $h$ is proportional to $h^2$.

\begin{lem}
  \label{raycount}
  Let the number of rays $\rho$ such that $\left|\rho\right|\leq h$ be $N_h $. Then \[\lim_{h\rightarrow \infty} \frac{N_h}{h^2}=\frac{4}{\zeta(2)}\approx 2.43.\] where $\zeta$ is the Riemann zeta function.
\end{lem}
\begin{proof}
  For each $\rho$, we can consider the lattice point $u_{\rho}$. Let $u_{\rho}=(x,y)$. Then we have $\gcd(x,y)=1$. This gives a bijective correspondence between pairs of integers $(x,y)$ with $\gcd(x,y)=1$, and rays $\rho$. Restricting to the case where $x$ and $y$ are both positive reduces the problem to counting the proportion of pairs of positive integers $(x,y)$ with $x,y\leq h$ such that $\gcd(x,y)=1$. Asymptotically, this proportion is given by $\frac{1}{\zeta(2)}$. Finally, since the restriction to $x,y>0$ simply counts the first quadrant, there are $4$ times as many rays overall. Thus we have \[\lim_{h\rightarrow \infty} \frac{N_h}{h^2}=\frac{4}{\zeta(2)}\]

\end{proof}

\begin{lem}
  \label{completesmooth}
  $\Sigma_h$ is smooth.
\end{lem}
\begin{proof}
  Recall that by Proposition~11.1.2 in \cite{torvar} it suffices to check that the variety is smooth at each of the torus fixed points. Furthermore, by Proposition~1.2.16 in \cite{torvar}, we can phrase smoothness as the area of the fundamental parallelogram formed by the two vectors corresponding to the rays forming the fans. By Pick's Theorem, this parallelogram has area $1$ if and only if the triangle formed by its fundamental rays contains no lattice points aside from the verticies.

  Now take a cone $\sigma$ formed by rays $\rho,\tau$. Suppose there existed some non-zero lattice point in the triangle formed by $u_{\rho}$ and $u_{\tau}$. Let this lattice point be $v$; since $v$ is in the triangle formed by $u_{\rho}$ and $u_{\tau}$, it is clear that $\left|v\right|<max\{|u_{\rho}|,|u_{\tau}|\}\leq h$. Consider the ray $\gamma = span(v)$. Since $\left|\gamma\right|\leq h$, $\gamma$ must be in the fan $\Sigma$. But by construction $\gamma\subset \sigma$. so $\gamma$ must be one of the two boundary rays. Thus the only lattice points on the triangle are $u_{\rho}$, $u_{\tau}$, and $0$. Thus every cone in $\Sigma_h$ is smooth which implies $\Sigma_h$ is smooth.

\end{proof}

To see this in the case of $h=3$, see Figure~\ref{fig:complete}. As noted, each of the triangles formed has area $1/2$ since they contain no interior points.

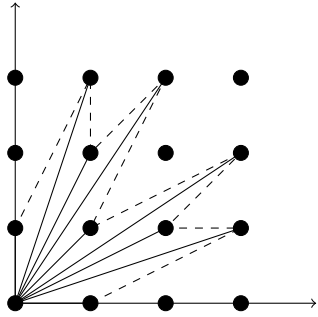
\begin{figure}
  \begin{center}
  \begin{tikzpicture}
    \foreach \x in {0,1,...,3}{
      \foreach \y in {0,1,...,3}{
        \node[draw,circle,inner sep=2pt,fill] at (\x,\y) {};
      }
    }
    \draw[dashed] (1,0) -- (3,1) -- (2,1) -- (3,2) -- (1,1) -- (2,3) -- (1,2) -- (1,3) -- (0,1);
    \draw (0,0) -- (1,0);
    \draw (0,0) -- (3,1);
    \draw (0,0) -- (2,1);
    \draw (0,0) -- (3,2);
    \draw (0,0) -- (1,1);
    \draw (0,0) -- (2,3);
    \draw (0,0) -- (1,2);
    \draw (0,0) -- (1,3);
    \draw (0,0) -- (0,1);
    
    \draw[->] (0,0) -- (0,4);
    \draw[->] (0,0) -- (4,0);
    
  \end{tikzpicture}
  \end{center}
  \caption{
    Here the solid lines represent the rays in the first quadrant of the fan $\Sigma_3$. Each of the triangles formed by a pair of neighboring rays, whose third side is show by a dashed line, contains no interior points. This gives us the smoothness of $\Sigma_3$.}
  \label{fig:complete}
\end{figure}

\section{Singularities From a Single Blowdown}
\label{sec:sing}

In the dense case, it suffices to consider those singularities that come from removing a single ray from the complete fan $\Sigma_h$. This is exactly the process that gives the blowdown of a toric surface along a ray. From this observation, we define the following two sets.

\begin{definition}
  Let $X=X(\Sigma_h)$ and let $\phi_{\rho}$ be the map that blows down along the ray $\rho$. We define
  \begin{itemize}
    \item $\displaystyle S_k=\left\{\rho \mid \phi_{\rho}\left(X\right)\text{ has a singularity of index }k\right\}$
    \item $\displaystyle S_{\geq k}=\left\{\rho \mid \phi_{\rho}\left(X\right)\text{ has a singularity of index of at least }k\right\}$
  \end{itemize}
\end{definition}

The following property is crucial for our main result.

\begin{cor}
  \label{singdensity}
  For $k\geq 1$ The set $S_{\geq k}$ has positive density in the limit $h\rightarrow \infty$.
\end{cor}

\begin{proof}
  Applying Proposition~\ref{bounds} gives that if $\frac{|\rho|}{h}<\frac{2}{k+2}$ then $\rho$ has singularity index at least $k$, and furthermore the set $\left\{\rho \middle| \frac{|\rho|}{h}<\frac{2}{k+2}\right\}$ has density $\left(\frac{2}{k+2}\right)^2$ inside of the set of rays. For some numeric tests on the actual density of such rays see Section~\ref{sec:dist}.
\end{proof}

This follows from the following stronger result which not only gives us the positive density condition but tells us where the lattice points corresponding to rays of a particular singularity index live.

\begin{prop}
  \label{bounds}
  Given any singularity index $k$, there exists $\epsilon=\epsilon(h)$, with $\epsilon\rightarrow 0$ as $h\rightarrow \infty$, such that $\rho\in S_k$ implies $\frac{2-\epsilon}{k+2}h \leq |u_{\rho}| \leq \frac{2}{k}h$
  
\end{prop}

For a visualization of the geometric implications of this proposition see Figure~\ref{fig:bands}.

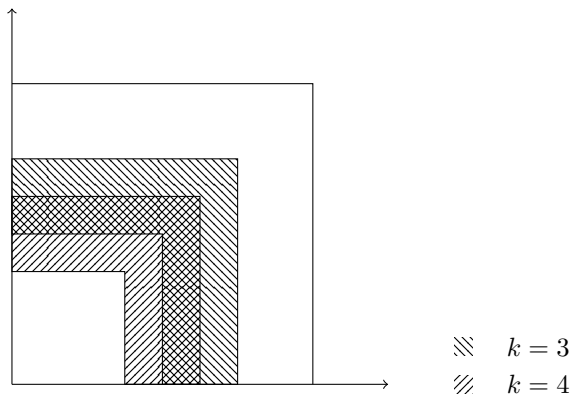
\begin{figure}
  \begin{center}
  \begin{tikzpicture}
    \draw (0,4)--(4,4)--(4,0);
    \draw[pattern=north west lines] (3,3)--(0,3)--(0,2)--(2,2)--(2,0)--(3,0)--cycle;
    \draw[pattern=north east lines] (2.5,2.5)--(0,2.5)--(0,1.5)--(1.5,1.5)--(1.5,0)--(2.5,0)--cycle;
    
    \draw[->] (0,0) -- (0,5);
    \draw[->] (0,0) -- (5,0);

    \node[rectangle,pattern=north west lines] at (6,0.5) {};
    \node at (7,0.5) {$k=3$};
    \node[rectangle,pattern=north east lines] at (6,0) {};
    \node at (7,0) {$k=4$};
    
  \end{tikzpicture}
  \end{center}
  \caption{For each self-intersection number, we expect that all rays with that self-intersection number lie in a band between $\frac{2-\epsilon}{k+2}h$ and $\frac{2}{k}h$. Although the bands overlap, the regions corresponding to higher self-intersection numbers are towards the interior.}
  \label{fig:bands}
\end{figure}

\begin{proof}
  Let us consider the behavior of the singularity index of a ray as we increase $h$. If we have a ray $\rho$, with neighboring rays $\tau $ and $\omega$, then by Theorem~10.4.4 in \cite{torvar} the singularity index is the unique integer $k$ such that $ku_{\rho}=u_{\tau}+u_{\omega}$. Rearranging that gives $\left|u_{\tau}+u_{\omega}\right|/k=|u_{\rho}|$. Then we apply the triangle inequality to give
  \[\frac{2h}{k}\geq \left|u_{\rho}\right|.\]

  Now we want a lower bound on $|u_{\rho}|$. For fixed $h$, the angle between neighboring rays in $\Sigma_h$ is bounded from above. Furthermore, this bound goes to zero as $h$ increases. Consequently there exists $\epsilon>0$ with $\epsilon \rightarrow 0$ as $h\rightarrow \infty$ such that
  \[\left|\frac{u_{\tau}}{|u_{\tau}|}-\frac{u_{\omega}}{|u_{\omega}|}\right|\leq \epsilon.\]
  
  Now we proceed by rewriting the left hand side and applying the triangle inequality
  \begin{align*}
    \left|\frac{|u_{\omega}|}{|u_{\tau}|}u_{\tau}-u_{\omega}\right| &= \left|\frac{|u_{\omega}|}{|u_{\tau}|}u_{\tau}+u_{\tau}-(u_{\omega}+u_{\tau})\right|\\
    &\geq \left|\left|\frac{|u_{\omega}|}{|u_{\tau}|}u_{\tau}+u_{\tau}\right|-\left|ku_{\rho}\right|\right|\\
    &\geq |u_{\tau}|+|u_{\omega}|-|ku_{\rho}|.
  \end{align*}
  Solving for $\left|ku_{\rho}\right|$ and then simplifying gives
  \begin{align}
    |ku_{\rho}|&\geq |u_{\tau}|+|u_{\omega}|-\left|\frac{|u_{\omega}|}{|u_{\tau}|}u_{\tau}-u_{\omega}\right|\nonumber \\
    &= |u_{\tau}|+|u_{\omega}|-|u_\omega|\cdot \left|\frac{u_{\tau}}{|u_{\tau}|}-\frac{u_{\omega}}{|u_{\omega}|}\right|\nonumber \\
    &\geq |u_{\tau}|+|u_{\omega}|-h\epsilon. \label{eq:1}
  \end{align}

  Now if we can find good lower bounds on $|u_{\tau}|$ and $|u_{\omega}|$ they will give a lower bound on $\left|u_{\rho}\right|$. Then since the situation for the left and right neighbors is symmetric it suffices to consider $\left|u_{\tau}\right|$.
  Starting with the case where $h=\left|u_{\rho}\right|$. Then in $\Sigma_h$ there is some $\gamma$ a neighbor of $\rho$. Now define a sequence of integer vectors $w_s$ by $w_s=u_{\gamma}+s\cdot u_{\rho}$.

\begin{claim*}
  Given any $h$, $\nu$ is a neighbor on the left or right of $\rho$ in $\Sigma_h$ if and only if $u_{\nu}=w_a$ for some $a$ with $w_{s}\notin \Sigma_h(1)$ for $s>a$.
\end{claim*}

\begin{proof}

  Consider the cone formed by $u_\rho$ and $w_0$. Since any neighbor must be at least as close to $\rho$ as $\gamma$, given $\rho\in\Sigma_h$, the neighbor of $\rho$ in $\Sigma_h$ must be in this cone. Thus the neighbor of $\rho$ is of the form $au_{\rho}+bw_0$ for some $a,b\in \NN$. Then since in every complete fan $\Sigma_h$, every cone is smooth, the cone formed by $au_{\rho}+bw_0$ and $\rho$ must be smooth. Thus we have that $\left|u_{\rho} \wedge (au_{\rho}+bw_0)\right|=1$. We can then simplify the left hand side.
  \begin{align*}
    \left|u_{\rho}\wedge(au_{\rho}+bw_0)\right|&=\left|u_{\rho}\wedge bw_0\right|\\
    &=b.
  \end{align*}
  This implies $b=1$. Furthermore, note that if $s>a$, then $su_{\rho}+bw_0$ is closer to $\rho$. However $au_{\rho}+w_0$ is the neighbor of $\rho$, thus $su_{\rho}+bw_0$ must not be in the fan $\Sigma_h$ for $s>a$.

\end{proof}

  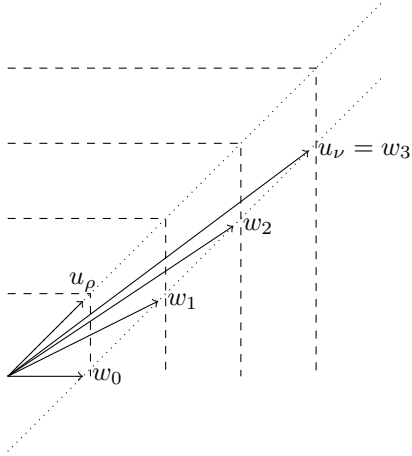
\begin{figure}
    \begin{center}
    \begin{tikzpicture}
      \coordinate (u) at (2,1);
      \coordinate (v) at (1,1);
      
      \draw[->] (0,0) -- (u) node[right] {$w_1$};
      \draw[->] (0,0) -- (v) node[above] {$u_{\rho}$};
      \draw[->] (0,0) -- ( $ (u)-(v) $ ) node[right] {$w_0$};
      \draw[->] (0,0) -- ( $ (u)+(v) $ ) node[right] {$w_2$};
      \draw[->] (0,0) -- ( $ (u)+(v)+(v) $ ) node[right] {$u_{\nu}=w_3$};

      \draw[dotted]  ($(u)-(v)-(v)$) -- ( $ (u)+(v)+(v)+(v) $ );
      \draw[dotted]  (0,0) -- ( $ (v)+(v)+(v)+(v)+(v) $ );

      \draw[dashed] (0,1.1) -- (1.1,1.1) -- (1.1,0);
      \draw[dashed] (0,2.1) -- (2.1,2.1) -- (2.1,0);
      \draw[dashed] (0,3.1) -- (3.1,3.1) -- (3.1,0);
      \draw[dashed] (0,4.1) -- (4.1,4.1) -- (4.1,0);
    \end{tikzpicture}
    \end{center}
        \caption{
          As a consequence of the claim, the nearest neighbors to $u_{\rho}$ at every possible value for $h$ lie along a line parallel to $u_{\rho}$.
          \label{fig:neighbor}
        }
  \end{figure}

  Now returning to the main proof, it remains to apply the claim to give a bound on $|u_{\tau}|$. To start let $u_{\tau}=w_a$, then 
  \begin{align*}
    h&\leq \left|w_{a+1}\right|\\
    &\leq\left|w_{a}+u_{\rho}\right|\\
    &\leq \left|u_{\tau}\right|+\left|u_{\rho}\right|.
  \end{align*}
  This can then be rearranged to give $|u_{\tau}|\geq h-|u_{\rho}|$. Repeating the process with $\omega$ yields $|u_{\omega}|\geq h-|u_{\rho}|$. This can then be substituted into Equation~\ref{eq:1}

  \begin{align*}
    |ku_{\rho}|&\geq |u_{\tau}|+|u_{\omega}|-h\epsilon\\
    |ku_{\rho}|&\geq 2\left(h-|u_{\rho}|\right)-h\epsilon\\
    |ku_{\rho}|&\geq 2h-2|u_{\rho}| - h\epsilon\\
    (k+2)\left|u_\rho\right|&\geq h(2-\epsilon)\\
    |u_{\rho}|&\geq \frac{2-\epsilon}{k+2}h.
  \end{align*}

  So given $\rho$ with singularity index $k$, the height of $\rho$, as a proportion of $h$, satisfies the inequality \[\frac{2-\epsilon}{k+2} < \frac{|u_{\rho}|}{h} < \frac{2}{k}.\]

\end{proof}

\begin{remark}
If we ignore the $\epsilon$, which goes to zero in the limits we care about, Proposition~\ref{bounds} tells us that for $\left|u_{\rho}\right|$ between $\frac{2}{k+1}h$ and $\frac{2}{k}h$, the only possible singularity indices are $k$ and $k+1$. See Figure~\ref{fig:space} in Section~\ref{sec:dist} to see how this affects the distribution of those minimal ray generators corresponding the singular rays.
\end{remark}

\section{Thresholds}
\label{sec:threshold}

\subsection{Main Result}

  Now we begin to piece the previous propositions together into a proof of our main result. To start though, we need a minor lemma.
  
\begin{lem}
  \label{limit}
  Let $q(n)$ be a function taking values in $\left[0,1\right]$ Then 
  \[\lim_{n\rightarrow \infty } \left(1-q(n)\right)^n=\exp\left(\lim_{n\rightarrow \infty}-nq(n)\right).\]
\end{lem}
Note in particular that this is simply a generalization of the familiar fact that $\lim_{x\rightarrow \infty} \left(1+\frac{1}{x}\right)^x=e$ by taking $q(n)=-1/n$.

\begin{proof}

  We start by taking logarithms and simplifying.
  \begin{align*}
    \lim_{n\rightarrow \infty}\log\left(1-q(n)\right)^n&=\lim_{n\rightarrow \infty}n\log\left(1-q(n)\right)\\
    &=\lim_{n\rightarrow \infty}n\left(-\sum_{k=1}^{\infty}\frac{q(n)^k}{k}\right)\\
    &=\lim_{n\rightarrow \infty}-nq(n)\left(\sum_{k=1}^\infty \frac{q(n)^{k-1}}{k}\right)\\
    &=\lim_{n\rightarrow \infty}-nq(n)\left(1 + q(n)\sum_{k=2}^\infty \frac{q(n)^{k-2}}{k}\right)\\
    &=\lim_{n\rightarrow \infty}-nq(n)+\lim_{n\rightarrow\infty}n\cdot (q(n))^2\sum_{k=2}^\infty \frac{q(n)^{k-2}}{k}\\
    &=\lim_{n\rightarrow \infty}-nq(n).
  \end{align*}
  
  In the case where $\lim_{n\rightarrow \infty}q(n)=0$ then the last step proceeds by the convergence of the sun. Otherwise, if $\lim_{n\rightarrow \infty}q(n)\neq 0$, we will have $\lim_{n\rightarrow \infty} nq(n) = \infty$ and $\lim_{n\rightarrow \infty} \log(1-q(n))^n =\infty$. Thus regardless of which case occurs \[\lim_{n\rightarrow \infty}\log\left(1-q(n)\right)^n=\lim_{n\rightarrow \infty}-nq(n).\]

  Finally since $\exp$ is continuous, we can apply $\exp$ to both sides.

  \begin{align*}
    \exp\left(\lim_{n\rightarrow \infty}\log\left(1-q(n)\right)^n\right)&=\exp\left(\lim_{n\rightarrow \infty}-nq(n)\right)\\
    \lim_{n\rightarrow \infty}\left(1-q(n)\right)^n&=\exp\left(\lim_{n\rightarrow \infty}-nq(n)\right).
  \end{align*}

\end{proof}

We can now proceed with the proof of the main theorem.

\begin{proof}[Proof of Theorem~\ref{mainresult}]

  For statement 1, let us start with the $q\prec 1/h^2$ case. Recall that the complete fan $\Sigma_h$ is smooth due to Lemma~\ref{completesmooth}. So it suffices to show that for $q\prec \frac{1}{h^2}$ the resulting fan is $\Sigma_h$ with high probability. Expanding the definition of $\prec$  gives  $h^2q\rightarrow 0$. Let $n$ be the number of rays with magnitude at most $h$. Then Lemma~\ref{raycount} along with the assumption $q\prec 1/h^2$ implies that $nq\rightarrow 0$. Then the probability of getting $\Sigma_h$ is given by \[P(\Sigma=\Sigma_h)=(1-q)^n.\] So we should compute $\lim_{h\rightarrow \infty} (1-q)^n$. By Lemma~\ref{limit}, we know that \[\lim_{h\rightarrow \infty} (1-q)^n=\exp(\lim_{h\rightarrow \infty} -nq)=\exp(0)=1.\] Thus in this case the fan is $\Sigma_h$ with high probability. Then for the $1-q\prec 1/h^2$ case by the same computation the fan is the empty fan with high probability. And the empty fan corresponds to $(\CC^*)^2$, which is smooth.

  For statement 2, recall that there are two cases, a dense case and a sparse case.
  
  \noindent \emph{Dense case}: $q\succ 1/h^2$ and $1-q\succ 1/h$
  
  By Corollary~\ref{singdensity} the rays which blow down to a singularity of index $k$ occur with positive density. So let this density be $\delta$, then there exists at least $m=\frac{\delta}{4}h^2$ of them such that none of them share any neighboring rays in $\Sigma_h$. Then for each ray, we can consider the cone formed by taking both of its neighbors.

  Each of these cones occurs in a fan in $T(h,1-q)$ with probability given by $(1-q)^2q$. Since none of these cones share any rays, we can consider them independently. Thus the probability that at least one of these cones is in our fan is given by $1-(1-(1-q)^2q)^{m}$. If at least one of these cones is present, then the fan has a singularity of size at least $k$, so it suffices to consider the following limit
  \[\lim_{h\rightarrow \infty}1-\left(1-\left(1-q\right)^2q\right)^{m}.\]
  We can simplify using Lemma~\ref{limit}.
  \begin{equation}
    \lim_{h\rightarrow \infty}1-\left(1-\left(1-q\right)^2q\right)^{m}=1-\exp\left(\lim_{h\rightarrow \infty}\left(1-q\right)^2qm\right).\label{eq:2}
  \end{equation}
  
  There are then two cases to consider.

  \noindent\emph{ Case $\lim_{h\rightarrow \infty}q\neq 1$}:

  In this case, since we have that $q\succ 1/h^2$, we also have
  \begin{align*}
    \lim_{h\rightarrow \infty} (1-q)^2qm&=\lim_{h\rightarrow \infty} (1-q)^2 \lim_{h\rightarrow \infty} qm\\
    &= \infty.
  \end{align*}

  \noindent\emph{ Case $\lim_{h\rightarrow \infty}q\neq 0$}:
  
  In this case we have $1-q\succ 1/h$ so
    \begin{align*}
      \lim_{h\rightarrow \infty} (1-q)^2qm&=\lim_{h\rightarrow \infty} (1-q)^2m \lim_{h\rightarrow \infty} q\\
      &= \infty.
    \end{align*}
    Now substituting into Equation~\ref{eq:2} gives \[\lim_{h\rightarrow \infty}1-(1-(1-q)^2q)^{m}=1.\] Thus with high probability the fan contains at least one singularity of index at least $k$.

    \noindent \emph{Sparse Case}: $1/h^2\prec (1-q)$ and $q\rightarrow 1$
    
    Given any ray $\rho$ there exists a set of rays in $\Sigma_h(1)$ with which, $\rho$ can form a smooth cone. Let $\tau$ be any ray such that the cone formed by $\rho$ and $\tau$ is smooth. Then given coordinates $(x,y)=u_{\rho}$ and $(a,b)=u_{\tau}$ the smoothness of the cone implies $\left|xb-ya\right|=1$. This implies that the lattice points corresponding to the rays with this property lie along two lines parallel to $\rho$, one on either side. Since the height of our rays is bounded above by $h$, there are at most $\frac{2h}{\left|\rho\right|}$ possible rays on either side.

    Now consider the set of rays for which there are at most 4 rays on either side with which we can make a smooth cone. It suffices to ask for $\frac{2h}{\left|\rho\right|}\leq 4$ which becomes $|\rho|\geq \frac{h}{2}$. Let this set of rays be $S$. For reasons of independence, we will want to restrict to the upper half plane, so let $S_{+}\subset S$ be the subset of those which are in the upper half plane. Finally we give $S_+$ an ordering starting the positive $x$-axis and proceeding counter clockwise.

    Let $\rho\in S_{+}$ be the $k$-th ray in $S_+$. Define $F_{\rho}$ to be the set of fans where $\rho$ is the first ray from $S_{+}$ and the cone containing $\rho$ as its rightmost ray is singular. Then since the first ray is unique, these are disjoint sets for distinct $\rho$. Thus we have

    \[P(\Sigma \in F_{\rho})\leq q^{k-1}(1-q)q^{4}\left(1-q^{n-4}\right).\]

    Here the $q^{k-1}(1-q)$ gives the probability that $\rho$ is the first ray. The $q^{4}$ gives the probability that none of the rays with which it could make a smooth fan exists. The $1-q^{n-4}$ gives the probability that at least one ray exists on the half plane after $\rho$ so that there is actually a cone that involves $\rho$. Now we sum each of these disjoint events.

    \[P(\Sigma\text{ is singular}) \geq \sum_{\rho\in S_+}P(\Sigma\in F_{\rho}).\]
    Now substituting the bound on $P(\Sigma\in F_{\rho})$ and observing that the resulting sum is a geometric series yields the following.
    \begin{align*}
      P(\Sigma\text{ is singular}) &\geq \sum_{k=1}^{\left|S_+\right|}q^{k-1}(1-q)q^4\left(1-q^{n-4}\right)\\
      &=\frac{1-q^{\left|S_+\right|}}{1-q} (1-q)q^4\left(1-q^{n-4}\right)\\
      &=\left(1-q^{\left|S_+\right|}\right)q^4\left(1-q^{n-4}\right).
    \end{align*}
    
    If $q\rightarrow 1$ then $q^4 \rightarrow 1$, and if $1/h^2\prec 1-q$ then $1-q^{n-4}\rightarrow 1$ and $1-q^{|S_+|}\rightarrow 1$. Thus for $q\rightarrow 1$ and $1/h^2\prec 1-q$, with high probability, $\Sigma$ is singular.

\end{proof}

As noted, the dense case and the sparse case proofs are entirely different, and reflect different ways in which the singularities arise. In the dense case, enough singularities arise from blowdowns along single rays that we can easily get the desired result from simply considering those singularities. This local computation contrasts with our proof in the sparse case. Here we look globally instead. For a fixed ray, there are only a fixed number of rays of height at most $h$ where the pair of rays form a smooth cone. Thus in many cases, we will fail to have any of these rays, and thus we will be forced to have a singular cone.

\subsection{Density}
\label{sec:density}

Let us define the density of singularities as the number of singular points divided by the number of fixed points. In particular given a fan $\Sigma$, let \[\delta_k\left(\Sigma\right)=\frac{\#\text{Points with singularity index } \geq k}{\# \text{Fixed Points}}.\] With this we can proceed to prove our density result.

\begin{proof}[Proof of Theorem~\ref{density}]
  Let $R_{\geq k}(\Sigma)$ be the number of singular points of index at least $k$ in $\Sigma$. We repeat the trick of considering certain cones as in the dense case of the proof of Theorem~\ref{mainresult}. As before, we can select $m:=\frac{\delta}{4}h^2$ of them, where $\delta$ is the density of rays with self-intersection number $k$ in $\Sigma_h$. Then $R_{\geq k}(\Sigma)$ is at least the number of such cones that exist.

  For $\Sigma\in T(h,1-q)$, the probability of $\Sigma$ having any particular cone of this form is given by $r:=(1-q)^2q$.  And since these cones are chosen to be independent, the number of such cones that exist is given by the binomial distribution, $\Bin(m,r)$. Thus we have

  \[P(R_{k}(\Sigma) \geq c\cdot h^2) \geq P(\Bin(m,r) \geq ch^2).\]

  Then we reverse the inequalities to give

  \[P(R_{k}(\Sigma) < c\cdot h^2) \leq P(\Bin(m,r) < ch^2).\]

  Now to bound the right hand side, we use a tail bound on the binomial distribution derived from Theorem 1 of \cite{hoeffding}. Viewing $\Bin(m,r)$ as the sum of $m$ identically distributed independent random variables each taking the value $1$ with probability $r$ and the value $0$ with probability $1-r$ gives the following inequality so long as $c\cdot h^2 \leq m\cdot r$.

  \[ P(\Bin(m,r) < c\cdot h^2)\leq e^{-2(mr-ch^2)^2/m}.\]
  
  Combining with the previous inequality gives
  
  \[P(R_{k}(\Sigma) < c\cdot h^2) \leq e^{-2(mr-ch^2)^2/m}.\]
  
  Now it suffices to show that the right hand side goes to zero. Since $m\sim h^2$ the previous formula can be rearranged to cancel the $m$ in the denominator yielding
  \[P(R_k(\Sigma)<c\cdot h^2)\leq e^{-2m(r-c)^2}.\]
  Since $r-c$ is constant and $m$ increases with $h$ the limit of the right hand side is zero. Thus
  \[\lim_{h\rightarrow \infty}P(R_k(\Sigma)<c\cdot h^2)=0.\]
  
  This gives a bound on the numerator of the density formula. To give a bound on the denominator for that formula notice that for sufficiently large $h$, given a fan $\Sigma$ the number of fixed points in  $\Sigma$ is less than the number of rays in $\Sigma$, and we know $\left|\Sigma(1)\right|\leq N_h\sim h^2$. Thus for sufficiently small $c$

  \[\lim_{h\rightarrow \infty}P(\delta_k(\Sigma)<c)=0.\]
  \end{proof}

A quick heuristic argument shows that so long as $p\rightarrow 1$, on average the density of cones in $\Sigma$ that are not in $\Sigma_h$ goes to zero, and since any singular points must come from cones not in the complete fan, from this we expect that with high probability the density will eventually be smaller than any positive number. Note also that this is distinct from asking that with high probability $\delta_k(\Sigma)=0$. Here perhaps by looking at the expected value of $\delta_k(\Sigma)$ we can find something more enlightening.

\section{Further Directions}
\label{sec:further}
\subsection{Direct Generalizations}

A first hope might be to generalize this to higher dimensions. But as stated before, our notion of a random fan does not directly extend to the case of higher dimensions. There are some potential candidates, but none seems particularly natural. The most straightforward way would simply be to still take random rays, but then use a ``minimal triangulation'' of the rays (viewed as points on the sphere) to choose a ``minimal simplicial completion'', but this process is now quite complex and hard to understand.

One might also try to consider the behavior along the threshold, namely the case where $q$ is proportional to $1/h^2$ or where $1-q$ is proportional to $1/h^2$. Numeric data suggests that in this case, the behavior is dependent only on the limit of $q\cdot h^2$ or $(1-q)h^2$.

\subsection{Complete Fans and the distribution of rays}
Instead of asking about random fans we can also ask about the properties of this complete fan that shows up in the smooth case. In particular, one might reasonably ask whether the limit of these fans has a reasonable interpretation as a geometric object. As a first step, we could try to compute what the density of rays in $S_{\geq k}$ is in the limit. This geometrically is counting the number of exceptional divisors with appropriate negative self-intersection number. The bounds in Proposition~\ref{bounds} suffice for the purposes of our result, but the numerical results presented in Figure~\ref{fig:limits} allow us to conjecture the following refinement.

\begin{conj}
  \label{exactbounds}
  For $k>1$
  \[\lim_{h\rightarrow \infty} \frac{\left|S_{\geq k}\right|}{\left|\Sigma_h(1)\right|} = \frac{2}{T_{k}}.\] Where $T_n=\frac{n^2+n}{2}$ is the n-th triangular number.
\end{conj}

\begin{figure}
\begin{tabular}{c|c|c|c|c|c|c|}
  & $\geq 2$ & $\geq 3$ & $\geq 4$ & $\geq 5$ & $\geq 6$ & $\geq 7$\\\hline
  10  & 0.6875 & 0.3125 & 0.1875 & 0.125 & 0.125 & 0.0625\\\hline
  50 & 0.6641 & 0.3359 & 0.1990 & 0.1370, & 0.0930 & 0.0724 \\\hline
  100 & 0.6675  & 0.3325 & 0.1997 & 0.1327 & 0.0959 & 0.0716 \\\hline
  500 & 0.6666 & 0.3334 & 0.2000 & 0.1334 & 0.0952 & 0.0714 \\\hline
  1000 & 0.6667 & 0.3333 & 0.2000 & 0.1333 & 0.0953 & 0.0714\\\hline
  5000 & 0.6667 & 0.3333 & 0.2000 & 0.1333 & 0.0952 & 0.0714\\\hline 
  Conjectured Limits & 2/3 & 2/6 & 2/10 & 2/15 & 2/21 & 2/28\\\hline
\end{tabular}

\caption{By considering the value $\left|S_{\geq k}\right|/\left|\Sigma(1)\right|$, we can compare the actual percentage of rays which have a particular self intersection number with the expected percentage as given by Conjecture~\ref{exactbounds}.\label{fig:limits}}
\end{figure}

In particular, this would imply that not only is $S_{\geq k}$ of positive density in $\Sigma_h(1)$ but also $S_{k}$, with that we would be able to strengthen Theorem~\ref{mainresult}.

\begin{conj}
  \label{stronger}
  Let $k>1$ then for a fan $\Sigma$ chosen with respect to $T(h,1-q)$. If $q\succ 1/h^2$ and $1-q\succ 1/h$ then with high probability $X(\Sigma)$ has a singularity of index of exactly $k$.
  If instead $1-q\succ 1/h^2$, then with high probability $X(\Sigma)$ has a singularity of index of at least $k$.
\end{conj}


For the first statement in this conjecture, it suffices to show that $S_k$ is of positive density. Notice that all we have is that $S_{\geq k}$ has positive density. Given this, the proof of the conjecture would proceed exactly as the proof of the dense case of statement 2 of Theorem~\ref{mainresult}.

For the second statement in this conjecture, a different technique to count singularities of a particular index is required, since we currently have no obvious way of doing this for cones that don't come from simply removing a single ray from $\Sigma_h$.

Also, there's no reason to believe that we can extend the first statement to cover the second case, since we expect that with so few cones in the sparse case, we would at times fail to have any singularities of any particular index. In particular, the number of fixed points for $1-q\prec 1/h$ grows slower than linearly, and we expect that the number of possible singularity indexes to grow linearly, so in general a fan would fail to achieve most of the possible singularity indexes.

\subsection{The Distribution of Rays}
\label{sec:dist}

Similarly, we can also ask about the distribution of these rays instead of their density. One way is to consider the distribution within $\ZZ^2$ with rays represented by their minimal lattice generators. This yields concentric shells of points with decreasing singularity index, as illustrated in Figure~\ref{fig:space} for the case of $h=5$. In particular, these shells can be a way to understand the geometry of Proposition~\ref{bounds} and Conjecture~\ref{exactbounds}.

\begin{figure}
  \begin{center}
  \begin{tikzpicture}
    \definecolor{color10}{rgb}{1,0,0}
    \definecolor{color9}{rgb}{0.8,0,0.1}
    \definecolor{color5}{rgb}{0.4,0,0.4}
    \definecolor{color4}{rgb}{0.3,0,0.5}
    \definecolor{color3}{rgb}{0.2,0,0.6}
    \definecolor{color2}{rgb}{0.1,0,0.7}
    \definecolor{color1}{rgb}{0,0,1}
    \begin{axis}[legend style={at={(1.2,1)},anchor=north}]
      \addplot+[only marks,mark=o] coordinates {
        (1, 5)
        (2, 5)
        (3, 5)
        (4, 5)
        (5, 4)
        (5, 3)
        (5, 2)
        (5, 1)
      };
      \addlegendentry{k=1}
      \addplot+[only marks,mark=square*] coordinates {
        (1, 4)
        (3, 4)
        (4, 3)
        (4, 1)
      };
      \addlegendentry{k=2}

      \addplot+[only marks,mark=*] coordinates {
        (1, 3)
        (2, 3)
        (3, 2)
        (3, 1)
      };
      \addlegendentry{k=3}
      \addplot+[only marks,mark=star] coordinates {
        (1, 2)
        (2, 1)
      };
      \addlegendentry{k=5}
      \addplot+[only marks,mark=diamond*] coordinates {
        (1, 1)
      };
      \addlegendentry{k=9}
      \addplot+[only marks,mark=square] coordinates {
        (1, 0)
        (0, 1)
      };
      \addlegendentry{k=10}

    \end{axis}
    
  \end{tikzpicture}
  \end{center}
  \caption{
    By plotting the values of the self intersection numbers in the case of $h=5$, we can see the pattern suggested by Proposition~\ref{bounds}. In particular, the rays with low self-intersection are concentrated near the boundary.
    \label{fig:space}
    }
\end{figure}
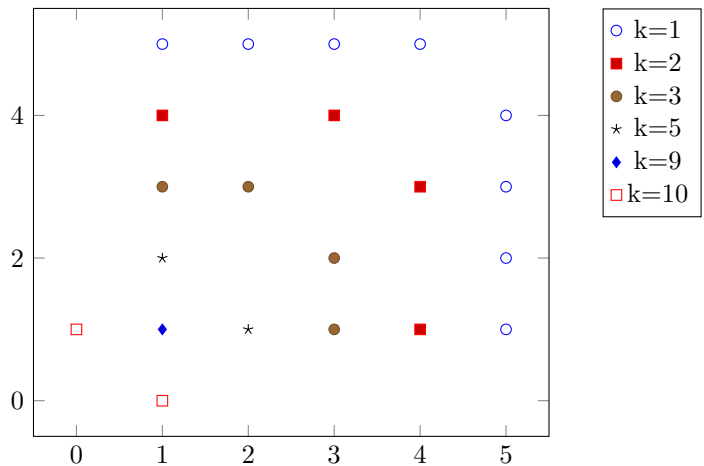

\section{Acknowledgments}

Thanks first to my adviser Daniel Erman for helping me edit this paper. Also to my friends for helping me find many of the typos in the paper. Thanks to the anonymous reviewer for the helpful comments in revising this paper. Thanks to the developers of Macaulay2\cite{M2} which I used initially to generate examples and test conjectures. Thanks also to Phillp Matchet Wood, for his assistance in understanding the binomial distribution. The author was partially supported by an NSF grant DMS-1502553.

\bibliography{smooth3}{}
\bibliographystyle{plain}

\end{document}